\documentclass[11pt]{article}
\usepackage{amssymb,amsmath,euscript}

\newtheorem{proposition}{Proposition}[section]

\newtheorem{lemma}[proposition]{Lemma}
\newtheorem{theorem}[proposition]{Theorem}
\newtheorem{definition}[proposition]{Definition}
\newtheorem{corollary}[proposition]{Corollary}

\def\l{{\langle}}
\def\r{\rangle}
\def\dim{{\rm dim}}

\def\R{{\mathbb R}}

\def\p{{\mathbb R}^p}
\def\d{{\mathbb R}^d}

\def\a{\alpha}
\def\be{\beta}
\def\la{\lambda}
\def\De{\Delta}
\def\de{\delta}
\def\ga{\gamma}

\def\ep{\varepsilon}

\def\si{\sigma}

\def\X{{\bf X}}

\def\E{{\mathbb E}}
\def\P{{\mathbb P}}

\def\N{{\mathbb R}^N}

\def\Var{{\rm Var}}
\def\Cov{{\rm Cov}}
\def\Gr{{\rm Gr}}
\def\msl{{\rm l.i.m.}}
\def\to{\rightarrow}

\def\({\big(}
\def\){\big)}

\makeatletter \@addtoreset{equation}{section} \makeatother
\newcommand {\qed}%
{%
    {}\hfill
    {}\hfill
    {$\square $}%
    \vspace {0.3cm}%
    \pagebreak [2]%
    \par
}%
\newenvironment{proof}[1]{%
    \vspace{0.3cm}%
    \pagebreak [2]%
    \par%
    \noindent {\bf  Proof~#1\ }}{\qed}%
\newenvironment{remark}{%
    \vspace{0.3cm} \pagebreak [2]%
    \par%
    \refstepcounter{proposition}
    \noindent%
   {\bf Remark~\theproposition\  }}{\ }%
\begin{document}
%


\title{Fractal and Smoothness Properties of Space-Time Gaussian Models}
\author{Yun Xue and Yimin Xiao \footnote{Research partially
        supported by  NSF grant DMS-0706728.}
\\Department of Statistics and Probability
\\Michigan State University\\}

\maketitle

\begin{abstract}

Spatio-temporal models are widely used for inference in statistics
and many applied areas. In such contexts interests are often in
the fractal nature of the sample surfaces and in the rate of
change of the spatial surface at a given location in a given
direction. In this paper we apply the theory of Yaglom (1957) to
construct a large class of space-time Gaussian models with
stationary increments, establish bounds on the prediction errors
and determine the smoothness properties and fractal properties of
this class of Gaussian models. Our results can be applied directly
to analyze the stationary space-time models introduced by Cressie
and Huang (1999), Gneiting (2002) and Stein
(2005), respectively.
\end{abstract}

{\sc Running head}: Fractal and Smoothness Properties of
Space-Time
Gaussian Models.\\

{\sc 2000 AMS Classification numbers}: 62M40, 62M30, 60G15, 60G17,
60G60, 28A80.\\

{\sc Key words:} Space-time models, anisotropic Gaussian fields,
prediction error, mean square differentiability, sample path
differentiability, Hausdorff dimension.

\section{Introduction}

Spatio-temporal models are widely used for inference in
statistics and many applied areas such as meteorology,
climatology, geophysical science, agricultural sciences,
environmental sciences, epidemiology, hydrology.
Such models presume, on $\d \times \R$, where $d$ is the spatial
dimension, a collection of random variables $X(x,t)$ at location
$x$ and time $t$. The family $\{X(x,t): (x,t)\in\d\times\R\}$ is
referred to as a spatio-temporal random field or a space-time
model.
Many authors have constructed various stationary space-time models
and the topic has been under rapid development in recent years.
See, for example, Jones and Zhang (1997), Cressie and Huang
(1999),  de Iaco, Myers and Posa (2001, 2002,
2003), Gneiting (2002), Gneiting, \emph{et al.} (2009),
Kolovos, \emph{et al.} (2004), Kyriakidis and
Journel (1999), Ma (2003a, 2003b, 2004, 2005a, 2005b, 2007, 2008),
Stein (2005) and their combined references for further
information on constructions of space-time models and
their applications.

There has also been increasing demand for non-stationary
space-time models. For example, in the analysis of spatio-temporal
data of environmental studies, sometimes there is little reason to
expect stationarity under the spatial covariance structures, and
it is more advantageous to have space-time models whose
variability changes with location and/or time. 
Henceforth, the construction of nonstationary space-time models
has become an attractive topic and several approaches have been
developed recently. These include to deform the coordinates
of an isotropic and stationary random field to obtain a rich class
of nonstationary random fields [see Schmidt and O'Hagan (2003),
Anderes and Stein (2008)], or to use convolution-based methods
[cf. Higdon, Swall and Kern (1999), Higdon (2002), Paciorek and
Schervish (2006), Calder and Cressie (2007)] or spectral methods
[Fuentes (2002, 2005)].

In this paper, we apply the theory of Yaglom (1957) to
construct a class of space-time Gaussian models with stationary
increments and study their statistical and geometric properties.
The main feature of this class of space-time models is that
they are anisotropic in time and space, and may have different
smoothness and geometric properties along different directions.
Such properties make them potentially useful as stochastic models
in various areas. By applying tools from Gaussian random fields,
fractal geometry and Fourier analysis, we derive upper and lower
bounds for the prediction errors, establish criteria for the
mean-square and sample path differentiability and determine
the Hausdorff dimensions of the sample surfaces, all in terms of
the parameters of the models explicitly. Our main results show
that the statistical and geometric properties of the Gaussian random
fields in this paper are very different from those obtained
by deformation from any isotropic random field. It is also
worth to mention that the method in this paper may be applied
to analyze more general Gaussian intrinsic random functions, 
convolution-based space-time Gaussian models [Higdon (2002), 
Calder and Cressie (2007)] and the spatial processes in Fuentes 
(2002, 2005).

The rest of this paper is organized as follows. In Section 2 we
construct a class of space-time Gaussian models with stationary
increments by applying the theory of Yaglom (1957). Then we
establish upper and lower bounds for the prediction errors
of this class of models in Section 3. In Section 4 we consider
smoothness properties of the models and establish explicit criteria
for the existence of mean square directional derivatives,
mean square differentiability and sample path continuity of
partial derivatives.  In Section 5 we look into the fractal
properties of these models and determine the Hausdorff dimensions
of the range, graph and level sets. In Section 6, we apply the main
results to some stationary space-time models, such as those
constructed by Cressie and Huang (1999), Gneiting (2002) and
Stein (2005). Finally, in Section 7, we provide proofs of
our results.

We end the Introduction with some notation. Throughout this paper,
instead of using space-time parameter space $\R^d \times \R$,
we take the parameter space as $\R^N$ or $\R^N_+ = [0,
\infty)^N$. We use $|\cdot|$ to denote the Euclidean norm
in $\R^N$. The inner product in $\R^N$ is denoted by
$\l \cdot, \cdot\r$. A typical parameter, $t\in\R^N$ is
written as $t = (t_1, \ldots, t_N)$. For any $s, t \in \R^N$
such that $s_j < t_j$ ($j = 1, \ldots, N$), $[s, t] =
\prod^N_{j=1}\, [s_j, t_j]$ is called a
closed interval (or a rectangle).

We will use $c, c_1, c_2, \ldots,$ to denote unspecified
positive and finite constants which may not be the same
in each occurrence.

\section{Anisotropic Gaussian models with stationary increments}

We consider a special class of {\it intrinsic random functions};
namely, space-time models with stationary increments. We will
further restrict ourselves to Gaussian random fields for which
powerful general Gaussian principles can be applied. Many of
the results in this paper can be extended to non-Gaussian
space-time models (such as stable or more general infinitely
divisible random fields), but their proofs require
different methods and go beyond the scope of this paper.
One can find some information for stable random fields
in Xiao (2008).

Throughout this paper, $X=\{X(t),t\in\N\}$ is a real-valued,
centered Gaussian random field with $X(0)=0$. We assume
that $X$ has stationary increments
and continuous covariance function $C(s,t)=\E[X(s)X(t)]$.
According to Yaglom (1957), $C(s,t)$ can be represented as
\begin{equation}\label{Eq:covar}
C(s,t)=\int_{\N}\(e^{i\l s,\la\r}-1\)\(e^{-i\l t,\la\r}-1\)F(d\la)
+\l s,Qt\r,
\end{equation}
where $Q$ is an $N\times N$ non-negative definite matrix and
$F(d\la)$ is a nonnegative symmetric measure on $\N\setminus\{0\}$
satisfying
\begin{equation}\label{Eq:int}
\int_{\N}\frac{|\la|^2}{1+|\la|^2}F(d\la)<\infty.
\end{equation}
In analogy to the stationary case, the measure $F$ is called
the \textit{spectral measure} of X. If $F$ is absolutely
continuous with respect to the Lebesgue measure in $\R^N$,
its density $f$ will be called the spectral density of $X$.

It follows from (\ref{Eq:covar}) that $X$ has the following
stochastic integral representation:
\begin{equation}\label{Eq:eks}
X(t)\stackrel {d} {=}\int_{\N}\(e^{i\l t,\la\r}-1\)W(d\la)+\l Y,t\r,
\end{equation}
where $X_1 \stackrel {d} {=} X_2$ means the processes $X_1$ and
$X_2$ have the same finite dimensional distributions, $Y$ is an
$N$-dimensional Gaussian random vector with mean 0 and covariance
matrix $Q$, $W(d\la)$ is a centered complex-valued Gaussian random
measure which is independent of $Y$ and satisfies
\[
\E\Big(W(A)\overline{W(B)}\Big)=F(A\cap B)\quad
 \textup{and}\quad W(-A)=\overline{W(A)}
\]
for all Borel sets $A, B\subseteq \N$. The spectral measure $F$ is
called the \textit{control measure} of $W$. Since the linear term
$\l Y,t\r$ in (\ref{Eq:eks}) will not have any effect on the
problems considered in this paper, we will from now on assume
$Y=0$. This is equivalent to assuming $Q=0$ in (\ref{Eq:covar}).
Consequently, we have
\begin{equation}\label{Eq:sigmaa}
v(h)\triangleq\E\big(X(t+h)-X(t)\big)^2=
2\int_{\N}\(1-\cos\l h,\la\r\)F(d\la).
\end{equation}
It is important to note that $v(h)$, called \textit{variogram} in
spatial statistics, is a negative definite function in the sense
of I. J. Schoenberg, which is determined by the spectral
measure $F$. See Berg and Forst (1975) for more information on
negative definite functions.

The above shows that various centered Gaussian random
fields with stationary increments can be constructed by
choosing appropriate spectral measures $F$. For the well
known fractional Brownian motion $B^H= \{B^H(t), t \in \R^N\}$
of Hurst index $H \in (0, 1)$, its spectral measure has a 
density function
\[
f_H(\la) = c(H, N) \frac 1 {|\la|^{2H + {N} }},
\]
where $c(H, N)>0$ is a normalizing constant such that
$v(h)= |h|^{2H}$. Since $v(h)$ depends on $|h|$ only,
$B^H$ is isotropic. Other examples of isotropic
Gaussian fields with stationary increments can
be found in Xiao (2007). We also remark that all
centered stationary Gaussian random fields can be treated using 
the above framework. In fact, if $Y = \{Y(t), t \in \R^N\}$
is a centered stationary Gaussian random field, it can
be represented as $Y(t)=\int_{\N}e^{i\l t,\la\r}\, W(d\la)$.
Thus the random field $X$ defined by
$$X(t)=Y(t)-Y(0)=\int_{\N}\(e^{i\l t,\la\r}-1\)\,W(d\la),
\quad \forall\ t \in \R^N
$$
is Gaussian with stationary increments and $X(0)=0$. Note
that the spectral measure $F$ of $X$ in the sense of
(\ref{Eq:sigmaa}) is the same as the spectral measure
[in the ordinary sense] of the stationary random 
field $Y$.

In the following, we propose and investigate
a class of centered, anisotropic Gaussian random
fields with stationary increments,
whose spectral measures are absolutely continuous
with respect to the Lebesgue measure in $\R^N$.
More precisely, we assume that the spectral measure
$F$ of $X=\{X(t), t\in \N\}$ is absolutely continuous
with density function $f(\la)$ which satisfies (\ref{Eq:int})
and the following condition:
\begin{itemize}
\item[(C)]\ There exist positive constants $c_1$,
$c_2$, $c_3 $, $\ga $ and $(\be_1,\cdots,\be_N)\in
(0,\infty)^N$ such that
\begin{equation}\label{Eq:leg}
\ga >\sum_{j=1}^N\frac{1}{\be_j}
\end{equation}
and
\begin{equation}\label{Eq:Maternb}
\frac{c_1}{\(\sum_{j=1}^N{|\la_j|}^{\be_j}\)^{\ga}}
\leq f(\la)\leq\frac{c_2}{\(\sum_{j=1}^N{|\la_j|}^{\be_j}\)^{\ga}},
\qquad  \forall \la\in \N \hbox{ with }\ |\la|\ge c_3.
\end{equation}
\end{itemize}

The following proposition  shows that (\ref{Eq:leg}) ensures
$f$ is a legitimate spectral density function.
\begin{proposition}\label{Le:int}
Assume that $f(\la)$ is a non-negative measurable function
defined on $\N$. If
$$\int_{|\la|\leq 1}|\la|^2 f(\la)d\la<\infty $$
and (\ref{Eq:Maternb}) holds,
then $f(\la)$ is a legitimate spectral density
if and only if the parameters $\ga$ and $\be_j$
for $j=1,\cdots,N$ satisfy (\ref{Eq:leg}).
\end{proposition}

Some remarks about condition (\ref{Eq:Maternb}) are in the following.
\begin{remark}
\begin{itemize}
\item There is an important connection between the random
field models that satisfy Condition (C) and those considered
in Xiao (2009). For $j=1,\cdots, N$, let
\begin{equation}\label{Eq:Hs}
H_j=\frac{\be_j}{2}\bigg(\ga-\sum_{i=1}^N \frac{1}{\be_i}\bigg)
\end{equation}
and let $Q=\sum_{j=1}^N \frac{1}{H_j}$. Then (\ref{Eq:Maternb})
can be rewritten as
\begin{equation}\label{Eq:aniMat}
\frac{c_4}{\(\sum_{j=1}^N{|\la_j|}^{H_j}\)^{2+Q}}
\leq f(\la)\leq\frac{c_5}{\(\sum_{j=1}^N{|\la_j|}^{H_j}\)^{2+Q}},
\qquad \forall \ \la\in \N \hbox{ with }\, |\la|\ge c_3,
\end{equation}
where the positive and finite constants $c_4$ and $c_5$
depend on $N$, $c_1, c_2$, $\beta_j$ and $\ga$ only.
To verify this claim, we will make use of the following
elementary fact: For any positive numbers $N$ and $q$,
there exist positive and finite constants $c_4$ and $c_5$
such that
\[
c_4 \bigg(\sum_{j=1}^N a_j\bigg)^{q}  \le \sum_{j=1}^N a_j^{q}
\le c_5 \bigg(\sum_{j=1}^N a_j\bigg)^{q}\]
for all non-negative numbers $a_1, \ldots, a_N$. Note that
\[
\bigg(\sum_{j=1}^N{|\la_j|}^{H_j}\bigg)^{2+Q} =
\bigg(\sum_{j=1}^N{|\la_j|}^{\beta_j \cdot
\frac 1 2 (\gamma - \sum_{i=1}^N \frac{1}{\be_i})}\bigg)^{2+Q}
\]
and $\frac 1 2 (\gamma - \sum_{i=1}^N \frac{1}{\be_i})(2+Q)
= \gamma$. We see that (\ref{Eq:Maternb}) and (\ref{Eq:aniMat})
are equivalent.

In turns out that the expression (\ref{Eq:aniMat}) is essential
in this paper and will be used frequently. For simplicity of
notation, from now on we take $c_3=1$.

\item It is also possible to consider Gaussian random
fields with stationary increments whose spectral measures
are not absolutely continuous. Some examples of such 
covariance space-time models can be found in Cressie and 
Huang (1999), Gneiting (2002), Ma (2003a, 2003b). Since 
the mathematical tools for studying such random fields are 
quite different, we will deal with them systematically in 
a subsequent paper.

\item Nonstationary Gaussian random fields can be constructed
through deformation of an isotropic Gaussian random field.
Refer to Anderes and Stein (2008) for more details. One of the
advantages of deformation is to closely connect a nonstationary
and/or anisotropic random field to a stationary and isotropic
one for which the existing statistical techniques are available.
However, there is also a disadvantage [from the point of
view of flexibility] associated with deformation.
Let $X(t)=Z(g^{-1}(t))$, where $\{Z(t), t \in \R^N\}$
is an isotropic Gaussian model and $g$ is a smooth bijection
of $\R^N$. Since the function $g$ is bi-Lipschitz
on compact intervals, the fractal dimensional properties of $X$
are the same as those of $Z$. Hence deformation of isotropic
Gaussian models will not generate anisotropic random fields
with rich geometric structures as shown by the models introduced
in this paper.
\end{itemize}
\end{remark}

\section{Prediction error of anisotropic Gaussian models}

Suppose we observe an anisotropic Gaussian random field $X$ on
$\N$ at $t^1,\ldots,t^n$ and wish to predict $X(u)$, for $u\in\N$.
Then the inference about $X(u)$ will be based upon the conditional
distribution of $X(u)$ given the observed values of
$X(t^1),\ldots,X(t^n)$. Refer to Stein (1999, Section 1.2)
for the closed form of this conditional distribution. A statistical
analysis typically aims at the optimal linear predictor of this
unobserved $X(u)$, known as \textit{simple kriging}. The simple
kriging predictor of $X(u)$ is
    \begin{equation}\label{Eq:krig}
     X^*(u)={\bf c}(u)^T\boldsymbol\Sigma^{-1}{\bf Z},
    \end{equation}
where ${\bf Z}=\(X(t^1),\ldots,X(t^n)\)^T$, ${\bf
c}(u)^T=\Cov\{X(u),{\bf Z}\}$ and $\boldsymbol\Sigma=\Cov({\bf
Z},{\bf Z}^T)$. The form (\ref{Eq:krig}) minimizes the mean square
prediction error, which then is given as $\Var (X(u))-{\bf
c}(u)^T\boldsymbol\Sigma^{-1}{\bf c}(u)$.
Since $X$ is Gaussian, the simple kriging is the conditional
expectation of $X(u)$ given ${\bf Z}$, and the mean square
prediction error is the conditional variance of $X(u)$
given ${\bf Z}$.

The main result of this section is Theorem \ref{Th:pred-error}
below, which gives lower and upper bounds for the mean
square prediction error for Gaussian random
fields with stationary increments which satisfy Condition (C).
It shows that, similar to stationary Gaussian field models
[cf. Stein (1999)], the prediction error of the models in
this paper only depends on the high frequency
behavior of the spectral density of $X$.

\begin{theorem}\label{Th:pred-error}
Let $X=\{X(t),t\in \N\}$ be a centered Gaussian random field
valued in $\R$ with stationary increments and spectral density
$f(\la)$ satisfying
(\ref{Eq:Maternb}). Then there exist constants $c_6>0$ and $c_7>0$,
such that for all integers $n\geq 1$ and all $u,t^1,\cdots,t^n\in \N,$\\
\begin{equation}\label{Eq:C3'}
c_6\min_{0\leq k\leq n}\sum_{j=1}^N|u_j-t^k_j|^{2H_j}
\leq\Var\(X(u)|X(t^1),\cdots,X(t^n)\)\leq
c_7\min_{0\leq k\leq n}\sum_{j=1}^N \sigma_j\big(|u_j-t^k_j|\big),
\end{equation}
where $H_j$ is given in (\ref{Eq:Hs}), $t^0=0$ and $\sigma_j: \R_+\to \R_+$
is defined by
\begin{equation}\label{Def:sigma}
\sigma_j(r) = \left\{\begin{array}{ll}
r^{2H_j}\qquad \quad &\hbox{ if }\ 0 < H_j < 1,\\
r^2|\log r| &\hbox{ if }\ H_j =1,\\
r^2 &\hbox{ if }\ H_j >1.
\end{array}
\right.
\end{equation}
\end{theorem}

If $H_j <1$, for $j=1,\cdots, N$, then the two bounds in
(\ref{Eq:C3'}) match. When there is some $H_j>1$, that means, the
random field $X(t)$ is smoother in the $j$-th direction
[see Corollaries \ref{Th:MSdiff}, \ref{Th:MSdiff2} and
Theorem \ref{Th:spdiff} below], then the
upper and lower bounds are not the same any more. This suggests
that the prediction error may become bigger as $X(t)$ gets smoother
in some directions.


The proof of Theorem \ref{Th:pred-error}, as well as those
of Theorems \ref{Th:moduli}, \ref{Th:dim} and \ref{Th:dim2}
reply partially on the following lemma, which provides upper
and lower bounds for the variogram of the model.

\begin{lemma}\label{Le:St}
Let $X=\{X(t), t\in \N\}$ be a centered Gaussian random field
valued in $\R$ with stationary increments and spectral density
$f(\la)$ satisfying (\ref{Eq:Maternb}). Then
there exist constants $c_{8}>0$ and $c_{9}>0$ such that
for $s,t \in \R^N$,
\begin{equation}\label{Eq:sigmab}
c_{8}\sum_{j=1}^N\sigma_j\big(|s_j-t_j|\big)
\leq \E\(X(s)-X(t)\)^2\leq
c_{9}\sum_{j=1}^N\sigma_j\big(|s_j-t_j|\big),
\end{equation}
where the function $\sigma_j$ is defined in (\ref{Def:sigma}).
\end{lemma}

The upper bound in (\ref{Eq:sigmab}) implies that $X$ has
a version whose sample functions are almost surely continuous.
Throughout this paper, without loss of generality, we will
assume that the sample function $t \mapsto X(t)$ is almost
surely continuous.

\section{Smoothness properties of anisotropic Gaussian models}
\label{sec:smoothness}

Regularity properties of sample path of random fields are of
fundamental importance in probability and statistics. Many
authors have studied mean square and sample path continuity
and differentiability of  Gaussian processes and fields. See
Cram\'er and Leadbetter (1967), Alder (1981), Stein (1999),
Banerjee and Gelfand (2003), Adler and Taylor (2007). In this
section we provide explicit criteria for the mean square
and sample path differentiability, for the models introduced
in Section 2.

\subsection{Distributional properties of mean square
partial derivatives}

Banerjee and Gelfand (2003) studied the smoothness
properties of stationary random fields and some
non-stationary relatives through
directional derivative processes and their distributional
properties. To apply their method to random fields
with stationary increments, let us first recall the definition
of mean square directional derivatives.

\begin{definition}
Let $u \in \N$ be a unit vector. A second order random field
$\{X(t), t\in \N\}$ has mean square directional
derivative $X'_u(t)$ at $t\in \N$ in the direction $u$ if, as $h\to 0$,
\[
 X_{u,h}(t)=\frac{X(t+h u)-X(t)}{h}
\]
converges to $X'_u(t)$ in the $L_2$ sense. In this case, we write
$X'_u(t) = {\rm l.i.m.}_{h \to 0} X_{u,h}(t)$.
\end{definition}

Let $e_1,e_2,\cdots,e_N$ be an orthonormal basis for $\N$. If $u = e_j$,
then $X'_{e_j}(t)$ is the mean square partial derivative in the $j$-th
direction defined in Adler (1981), which will simply be written as
$X'_{j}(t)$. We will also write $X_{e_j,h}(t)$ as $X_{j,h}(t)$.

For any second order, centered random field
$\{X(t), t\in \N\}$, similar to Theorem 2.2.2 in Adler (1981),
one can easily establish a criterion in terms of the
covariance function $C(s, t) = \E\big[X(t) X(s)\big]$ for
the existence of mean square directional derivative $X'_u(t)$.
Banerjee and Gelfand (2003) further showed that
the covariance function of $X'_u(t)$ is given by
\[
\begin{split}
K_u(s, t) &= \lim_{h\to 0}\lim_{k\to 0} \E\big[X_{u, h}(t)
X_{u, k}(s)\big ]\\
&=  \lim_{h\to 0}\lim_{k\to 0} \frac{C(t+hu, s + ku)- C(t+hu, s)-
C(t, s + ku)+C(t, s) } {hk}.
\end{split}
\]
Extending their argument, one obtains the following theorem for Gaussian
random fields with stationary increments.

\begin{theorem}\label{Th:diff}
Let $X=\{X(t), t\in\N\}$ be a centered Gaussian random field
valued in $\R$ with stationary increments, then the mean
square partial derivative $X'_{j}(t)$ exists for all
$t \in \R^N$ if and only if the limit
\begin{equation}\label{Eq:msd-nsc}
\lim_{h,k\to 0}\frac{v(he_j)+v(ke_j)-v\((h-k)e_j\)}{hk}
\end{equation}
exists. Moreover, this later condition is equivalent to
$v(t)$ has second-order partial derivatives at 0
in the $j$-th direction.
\end{theorem}

As a consequence, we obtain an explicit criterion for the
existence of mean square partial derivatives of Gaussian random 
fields in Section 2.

\begin{corollary}\label{Th:MSdiff}
Let $X=\{X(t),t\in \N\}$ be a centered Gaussian random field
valued in $\R$ with stationary increments and spectral density
$f(\la)$ satisfying Condition (C). Then for every $j=1,\cdots,N$,
the mean square partial derivative $X'_j(t)$ exists if and only if
\begin{equation}\label{Eq:mqdiff}
\be_j\bigg(\ga-\sum_{i=1}^N\frac{1}{\be_i}\bigg)>2,
\end{equation}
or equivalently $H_j > 1$ {\rm [}cf. (\ref{Eq:Hs}){\rm ]}.
\end{corollary}

Assume condition (\ref{Eq:msd-nsc}) of Theorem \ref{Th:diff} holds 
so that the mean square partial derivative $X'_j(t)$ exists for all
$t \in \R^N$. We now consider the distributional properties of the
random field $\{X'_j(t), t \in \R^N\}$. 

Since $\E(X(t))=0$
for all $t\in\N$, we have $\E(X_{j,h}(t))=0$ and $\E(X'_j(t))=0$. Let
$C_j^{(h)}(s,t)$ and $C_j(s,t)$ denote the covariance functions
of the random fields $\{X_{j,h}(t), t \in \R^N\}$ and $\{X'_j(t), 
t \in \R^N\}$, respectively. Let $\De=s-t$, we immediately have
\begin{equation}\label{Eq:covdi}
C_j^{(h)}(s,t)=\frac{v(\De+he_j)+v(\De-he_j)-2v(\De)}{2h^2},
\end{equation}
and $\Var(X_{j,h}(t))=v(he_j)/h^2$, which only
depends on the scalar $h$.

\begin{theorem}\label{Th:deriv}
Let $X=\{X(t), t \in \R^N\}$ be a centered Gaussian random
field valued in $\R$ 
with stationary increments. Suppose that all second-order
partial derivatives of the variogram $v(t)$ exist.
Then the covariance function of $X'_j(t)$ is given by
\begin{equation}\label{Eq:derivv}
C_j(s,t)=\frac{1}{2}v_j^{''}(s-t),
\end{equation}
where $v_j^{''}(t)$ is the second-order partial derivative
of $v$ at $t$ in the $j$-th direction. In particular, $\{X'_j(t),
t \in \R^N\}$ is a stationary Gaussian random field.
\end{theorem}

\begin{proof}\ The desired result follows from (\ref{Eq:covdi}).
\end{proof}

It is also useful to determine the covariance of $X(s)$ and $X'_j(t)$
for all $s,t\in\N$. Since
\[
\Cov\(X(s),X_{j,h}(t)\)=\frac{1}{2h}\,\Big\{v(t+he_j)
-v(t)+v(\De)-v(\De-he_j)\Big\},
\]
where $\De=s-t$, we obtain
\[
\Cov\(X(t),X_{j,h}(t)\)=\frac{1}{2h}\,\Big\{v(t+he_j)-v(t)-v(he_j)\Big\}
\]
and
\begin{equation}\label{Eq:crosscov}
\begin{split}
\Cov\(X(s),X'_j(t)\)&=\lim_{h\to0}\frac{1}{2h}\, \Big\{v(t+he_j)
-v(t)+v(\De)-v(\De-he_j)\Big\}\\
&=\frac{1}{2}\(v'_j(t)+ v'_j(\De)\),
\end{split}
\end{equation}
where $v'_j(t)$ is the partial derivative of $v$ at $t$ in the
$j$-th direction.

In particular, $\Cov\(X(t),X'_j(t)\)=v'_j(t)/2$, which is
different from the stationary case. Recall that if $Y(t)$
is a stationary Gaussian field with mean square partial
derivative $Y'_j(t)$, then $Y(t)$ and $Y'_j(t)$ are
uncorrelated. That means, the level of the stationary
random field at a particular location is uncorrelated
with the partial derivative in any direction at
that location. However, this is not always true for
nonstationary random fields.

Next we consider the bivariate process
\[Y_j^{(h)}(t)=\left(
                 \begin{array}{c}
                   X(t) \\
                   X_{j,h}(t) \\
                 \end{array}
               \right).
\]
It can be verified that this process has mean 0 and
cross-covariance matrix
\begin{equation*}\label{Vmat}
\begin{split}
&V_{j,h}(s,t)\\
&=\left(
\begin{array}{cc}
\displaystyle\frac{v(s)+v(t)-v(\De)}{2}
& \displaystyle\frac{v(t+h e_j)-v(t)+v(\De)-v(\De-h e_j)}{2h}  \\
\displaystyle\frac{v(s+h e_j)-v(s)+v(\De)-v(\De+h e_j)}{2h}
& \displaystyle\frac{v(\De+h e_j)+v(\De-h e_j)-2v(\De)}{2h^2}
\\
\end{array}
\right).
\end{split}
\end{equation*}
Because $Y_j^{(h)}(t)$ is obtained by linear transformation of $X(t)$,
the above is a valid cross-covariance matrix in $\N$. Since
this is true for every $h$, letting $h\to 0$ we see that
\[
V_j(s,t)=\left(
\begin{array}{cc}
\displaystyle\frac{1}{2}\big\{ v(s)+v(t)-v(\De)\big\}
& \displaystyle\frac{1}{2}\big\{v'_j(t)+v'_j(\De)\big\}\\
\displaystyle\frac{1}{2}\big\{ v'_j(s)-v'_j(\De)\big\}
& \displaystyle\frac{1}{2}v_j^{''}(\De) \\
\end{array}
\right)
\]
is a valid cross-covariance matrix in $\N$. In fact, $V_j$ is the
cross-covariance matrix for the bivariate process
\[
Y_j(t)=\left(
\begin{array}{c}
X(t) \\
X'_j(t) \\
\end{array}
\right).
\]

\subsection{Criterion for mean square differentiability}

Benerjee and Gelfand (2003) pointed out that the existence
of all mean square directional derivatives of
a random field $X$ does not even guarantee mean square
continuity of $X$, and they introduced a notion of mean
square differentiability which has analogous properties
of total differentiability of a function in $\N$ in the
non-stochastic setting. We first recall their definition.
\begin{definition}
A random field $\{X(t), t\in \N\}$ is mean square differentiable
at $t\in \N$ if there exists a (random) vector $\nabla_X(t) \in \N$ 
such that for all scalar $h>0$, all vectors $u \in {\cal S}_N =
\{t\in \N: |t|=1\}$
\begin{equation}\label{def:msdiff}
X(t+h u)=X(t)+h u^T\nabla_X(t)+r(t,h u),
\end{equation}
where $r(t,h u)/h\to0$ in the $L_2$ sense as $h\to 0$.
\end{definition}

In other words,  for all vectors $u \in {\cal S}_N$, it
is required that
\begin{equation}\label{Eq:msd}
\lim_{h \to 0} \E\bigg(\frac{X(t+h u)-X(t)-h u^T\nabla_X(t)}
{h}\bigg)^2=0.
\end{equation}
It can be seen that if $X$ is mean square differentiable
at $t$, then for all unit vectors $u \in {\cal S}_N$
\begin{equation*}
\begin{split}
X'_u(t)&=  \stackrel{\displaystyle\msl}{\scriptstyle{h\to 0}}
\frac{X(t+h u)-X(t)}{h}\\
&=\stackrel{\displaystyle\msl}{\scriptstyle{h\to 0}}
\frac{h u^T\nabla_X(t)+r(t, h u)}{h}\\
&=u^T\nabla_X(t).
\end{split}
\end{equation*}
Hence it is necessary that $\nabla_X(t)=(X'_1(t), \ldots,
X'_N(t))$. 

The next theorem provides a sufficient condition
for a Gaussian random field with stationary increments
to be mean square differentiable.

\begin{theorem}\label{Th:msdifferentiability}
Let $X=\{X(t),t\in \N\}$ be a centered Gaussian random field
valued in $\R$ with stationary increments. If all the
second-order partial and mixed derivatives of the
variogram $v(t)$ exist and are continuous, then $X$ is
mean square differentiable at every $t\in\N$.
\end{theorem}

As a consequence of Theorem \ref{Th:msdifferentiability} we
obtain
\begin{corollary}\label{Th:MSdiff2}
Let $X=\{X(t),t\in \N\}$ be a centered Gaussian random field
valued in $\R$ with stationary increments and spectral density
$f(\la)$ satisfying Condition (C). If
\begin{equation}\label{Eq:mqdiff2}
\be_j\bigg(\ga-\sum_{i=1}^N\frac{1}{\be_i}\bigg)>2
\  \qquad \hbox{for every } j=1,\ldots,N,
\end{equation}
then $X$ is mean square differentiable at every $t\in\N$.
\end{corollary}

\subsection{Criterion for sample path differentiability}

For many theoretical and applied purposes, one often needs
to work with random fields with smooth sample paths.
Refer to Adler (1981), Adler and Taylor (2007) and the
reference therein for more information. Since in general
mean square differentiability does not imply almost sure
sample path differentiability, it is of interest
to provide convenient criteria for the latter. 

For Gaussian random fields considered in this paper, it turns
out that under the same condition as Corollary \ref{Th:MSdiff},
the partial derivatives of $X$ are almost surely continuous.

\begin{theorem}\label{Th:spdiff}
Let $X=\{X(t),t\in \N\}$ be a separable and centered Gaussian
random field with values in $\R$. We assume that $X$ has
stationary increments and  satisfies Condition (C).\\
{\rm (i).} If
\begin{equation}\label{Eq:diff}
\be_j\bigg(\ga-\sum_{i=1}^N\frac{1}{\be_i}\bigg)>2
\qquad (i.e., H_j > 1),
\end{equation}
for some $j \in \{1, \cdots, N\}$,
then $X$ has a version $\widetilde{X}$ with continuous
sample functions such that its $j$th partial derivative
$\widetilde{X}'_j(t)$ is continuous almost surely.\\
{\rm (ii).} If (\ref{Eq:diff}) holds for all $j \in \{1, \cdots, N\}$,
then $X$ has a version $\widetilde{X}$ which is continuously
differentiable in the following sense: with probability 1,
\begin{equation}\label{def:sample-diff}
\lim_{h \to 0} \frac{\widetilde{X}(t+h u)- \widetilde{X}(t)
-h u^T\nabla_{\widetilde{X}}(t)} {h} = 0
\qquad \hbox{ for all } u \in {\cal S}_N
\ \hbox{ and } t \in \R^N.
\end{equation}
\end{theorem}

If condition (\ref{Eq:diff}) does not hold for some
$j \in \{1, \cdots, N\}$, then $X(t)$ does
not have mean square partial derivatives along
those directions and $X(t)$ is usually a
random fractal. In this case, it is of interest to characterize
the asymptotic behavior of $X(t)$ by its local and uniform
moduli of continuity.

These problems for anisotropic Gaussian random fields have been
considered in Xiao (2009) and the methods there are applicable
to $X$ with little modification. For completeness, we state
the following result which can be proved by using Lemma \ref{Le:St}
and general Gaussian methods. We omit its proof.

\begin{theorem}\label{Th:moduli}
Let $X =\{X(t),t\in \N\}$ be as in Theorem \ref{Th:spdiff}. Then for
every compact interval $I \subset \R^N$, there exists a
positive and finite constant $c_{10}$, depending only on $I$ and
$H_j,\, (j=1,\ldots,N)$ such that
\begin{equation}
\limsup_{|\ep|\to 0}\frac{\sup_{t\in I,\,s\in[0,\ep]}|X(t+s)-X(t)|}
{\sqrt{\varphi(\ep) \log(1+\varphi(\ep)^{-1})}}\leq c_{10},
\end{equation}
where $\varphi(\ep)=\sum_{j=1}^N\sigma_j(|\ep_j|)$
for all $\ep = (\ep_1, \ldots, \ep_N)\in \R^N$. 		
\end{theorem}

\section{Fractal properties of anisotropic Gaussian models}

The variations of soil, landform and geology are usually highly
non-regular in form and can be better approximated by a stochastic
fractal. Hausdorff dimension has been extensively used in
describing fractals. We refer to Kahane (1985) or Falconer
(1990) for their definitions and properties.

Let $X=\{X(t),t\in \N\}$ be a real-valued, centered
Gaussian random field. For any integer $p \ge 1$,
we define an $(N, p)$-Gaussian random field ${\bf X} =
\{{\bf X}(t), t \in \N\}$ by
\begin{equation}\label{def:X}
{\bf X}(t) = \big(X_1(t), \ldots, X_p(t)\big),\qquad t \in \R^N,
\end{equation}
where $X_1, \ldots, X_p$ are independent copies of $X$.

In this section, under more general conditions on $X$ than those
in Sections 2--4, we study the Hausdorff dimensions of the range
$\X([0,1]^N)=\{\X(t):t\in
[0,1]^N\}$, the graph $\Gr \X([0,1]^N)=\{(t,\X(t)):t\in [0,1]^N\}$
and the level set $\X^{-1}(x)=\{t\in \N: \X(t)=x\}$ ($x \in \R^p$).
The results in this section can be applied to wide classes of
Gaussian spatial or space-time models (with or without stationary
increments).

First, let's consider fractional Brownian motion
$B^H=\{B^H(t),t\in \N\}$ valued in $\p$ with Hurst index $H \in (0, 1)$.
$B^H(t)$ is a special example of our model which, however, has
isotropic spectral density.
It is known [cf. Kahane (1985)] that
\[\dim \,{\rm Gr}\,B^H\big([0,1]^N\big)=\min\bigg\{N+(1-H)p,\,\,
\frac{N}{H}\bigg\} \quad\quad {\rm a.s.}\]
Especially, when $p=1$,
\[
\dim\, {\rm Gr}\,B^H\big([0,1]^N\big)= N+1-H\quad\quad {\rm a.s.}
\]
Also,
\[
\dim \,(B^H)^{-1}(x)=N-H,\quad\quad {\rm a.s.}
\]
The fractal properties of fractional Brownian motion have
been applied by many statisticians to estimate the Hurst
index $H$ and it is sufficient to choose $p=1$. Refer to
Hall and Wood (1993), Constantine
and Hall (1994), Kent and Wood (1997), Davis and Hall
(1999), Chan and Wood (2000, 2004), Zhu and Stein (2002).

Let $(\overline{H}_1, \ldots, \overline{H}_N) \in (0, 1]^N$
be a constant vector. Without loss of generality,
we assume that they are ordered as
\begin{equation}\label{Eq:Hs2}
0 < \overline{H}_1 \le \overline{H}_2 \le \cdots \le
\overline{H}_N \le 1.
\end{equation}
We assume the following conditions.

{\rm (D1).}\ There exist positive constants $\de_0, \,c_{11}
 \ge 1$  such that for all $s, t \in [0, 1]^N$ with $|s-t|\le \delta_0$
\begin{equation}\label{Eq:CompareVar2}
c_{11}^{-1}\, \sum_{j=1}^N |s_j - t_j|^{2\overline{H}_j}
\le \E\bigl[\bigl(X(t)-
X(s)\bigr)^2\bigr] \le c_{11}\, \sum_{j=1}^N |s_j - t_j|^{2\overline{H}_j}.
\end{equation}

{\rm (D2).}\ For any constant $\varepsilon \in (0, 1)$, there exists
a positive constant $c_{12}$ such that for all $u, t$ $\in [\ep, 1]^N$, 
we have
\begin{equation} \label{Eq:slnd2}
{\rm Var} \left( \left. X(u)\, \right|\, X(t)\right) \ge c_{12}\,
\sum_{j=1}^N  \big| u_j - t_j \big|^{2\overline{H}_j}.
\end{equation}

The following theorems determine the Hausdorff dimensions of
range, graph and level sets of $\X$. Because of anisotropy,
these results are significantly different from the aforementioned
results for fractional Brownian motion or other isotropic random
fields [cf. Xiao (2007)]. Even though Theorems \ref{Th:dim} and
\ref{Th:dim2} below are similar to Theorems 6.1 and 7.1 in Xiao
(2009), they have wider applicability. In particular, they can
be applied to a random field $X$ which may be smooth in certain
(or all) directions.

\begin{theorem}\label{Th:dim}
Let $\X=\{\X(t),t\in \N\}$ be an $(N, p)$-Gaussian random field
defined by (\ref{def:X}). If the coordinate process $X$ satisfies
Condition (D1), then, with probability 1,
\begin{equation}\label{Eq:dim1}
\dim \,\X \big([0,1]^N\big)=\min\bigg\{p;\, \sum_{j=1}^N\frac{1}
{\overline{H}_j}\bigg\},
\end{equation}
and
\begin{equation}\label{Eq:dim2}
\dim\,{\rm Gr} \X\big([0,1]^N\big)=\min_{1\leq k\leq
N}\bigg\{\sum_{j=1}^k\frac{\overline{H}_k}{\overline{H}_j}
+N-k+(1- \overline{H}_k)p;\,\,
\sum_{j=1}^N\frac{1}{\overline{H}_j}\bigg\},
\end{equation}
where $\sum_{j=1}^0\frac{1}{\overline{H}_j}:=0$.
\end{theorem}

\begin{proof}\
The right inequality in (\ref{Eq:CompareVar2}) and
Theorem 4.9 show that $\X(t)$ satisfies a uniform H\"older
condition on $[0, 1]^N$ which, in turn, imply the desired
upper bounds in  (\ref{Eq:dim1}) and (\ref{Eq:dim2}).

The lower bounds for $\dim \,\X \big([0,1]^N\big)$ and $
\dim\,{\rm Gr} \X\big([0,1]^N\big)$ can be derived from the
left inequality in (\ref{Eq:CompareVar2}) and a capacity
argument. See the proof of Theorem 6.1 in Xiao (2009)
for details.
\end{proof}

For the level sets of $\X$, we have
\begin{theorem}\label{Th:dim2}
Let $\X=\{\X(t),t\in \N\}$ be an $(N, p)$-Gaussian random field
defined by (\ref{def:X}). If the coordinate process $X$ satisfies
Conditions (D1) and (D2), then the following statements hold:\\
{\rm (i)}\ If $\sum_{j=1}^N\frac{1}{\overline{H}_j}<p$,
then for every $x\in\p\backslash \{0\}$, $\X^{-1}(x)=\emptyset$ a.s.\\
{\rm (ii)}\ If $\sum_{j=1}^N\frac{1}{\overline{H}_j}>p$,
then for any $x\in\p$, with positive probability
\begin{equation} \label{Eq:leveldim}
\dim\, \X^{-1}(x)=
\min_{1\leq k\leq N}\bigg\{\sum_{j=1}^k\frac{\overline{H}_k}
{\overline{H}_j}+N-k-\overline{H}_k\, p \bigg\}.
\end{equation}
\end{theorem}

\begin{proof}\ The results (i) and (ii) follow
from the proof of Theorem 7.1 in Xiao (2009).
\end{proof}

Let $X=\{X(t),t\in \N\}$ be a centered Gaussian random field
valued in $\R$ with stationary increments and spectral density
$f(\la)$ satisfying (\ref{Eq:Maternb}). Let $H_1, \ldots, H_N$
be defined by (\ref{Eq:Hs}). Then, by Theorem
\ref{Th:pred-error} and Lemma \ref{Le:St}, we see that $X$
satisfies (D1) for all $\overline{H}_j \in
(0,\, 1 \wedge H_j)$ ($1 \le j \le N$). It also satisfies
Condition (D2) with $\overline{H}_j = H_j$ provided
$H_j \le 1$ for all $j=1, \cdots, N$. Hence one can apply
Theorems \ref{Th:dim} and \ref{Th:dim2} to derive the following
result.
\begin{corollary}\label{Co:SI-dim}
Let $\X=\{\X(t),t\in \N\}$ be a centered Gaussian random field
valued in $\p$ defined by (\ref{def:X}). We assume that its
coordinate process $X$ has stationary increments with
spectral density $f(\la)$ satisfying (\ref{Eq:Maternb})
and $H_j \ (j = 1, \cdots, N)$ defined by (\ref{Eq:Hs}) are
ordered as $H_1 \le H_2 \le \cdots \le H_N$. We have\\
{\rm (i).} With probability 1, (\ref{Eq:dim1}) and (\ref{Eq:dim2})
hold with $\overline{H}_j = 1 \wedge H_j$
($1 \le j \le N$).\\
{\rm (ii).} If, in addition, $H_j \le 1$ for all $j=1, \cdots, N$
and $\sum_{j=1}^N\frac{1}{H_j}>p$,
then (\ref{Eq:leveldim}) holds with positive probability.
\end{corollary}

We believe that the above fractal properties can also be
useful for estimating the parameters $H_1, \ldots, H_N$
of our model. However this will be more subtle than the
isotropic case, where only the single parameter is
involved, for the following two reasons. First, if a parameter
$H_j > 1$, then the sample function $X(t)$ is smooth in the
$j$th direction and the Hausdorff dimensions of $\X$ has nothing
to do with $H_j$. In other words, based on fractal dimensions,
a parameter $H_j$ can be explicitly estimated only when $H_j < 1$.

Secondly, if we let $p=1$,
then (\ref{Eq:dim2}) gives $\dim{\rm Gr} X\big([0,1]^N\big)
=N+1-\overline{H}_1$, which does not give any information about
the other parameters $H_2, \ldots, H_N$. This suggests that,
in order to estimate all the parameters of an {\it anisotropic}
random field model, one has to work with a {\it multivariate}
random field $\X$ as defined by (\ref{def:X}).

\section{Applications to some stationary space-time models}

The above results can be applied to the stationary space-time
Gaussian fields constructed by Cressie and Huang (1999), Gneiting
(2002), de Iaco, Myers, and Posa (2002), Ma (2003a, 2003b)
and Stein (2005).

\subsection{Stationary covariance models}

Extending the results of Cressie and Huang (1999), Gneiting (2002)
showed that, for $(x,t)\in\d\times\R$,
\begin{equation}\label{Eq:Cre}
C(x,t)=\frac{\si^2}{\(1+a|t|^{2\a}\)^{\be N/2}}
\exp\bigg(-\frac{c|x|^{2\ga}}{\(1+a |t|^{2\a}\)^{\be\ga}}\bigg),
\end{equation}
is a stationary space-time covariance function, where $\si>0$,
$a>0$, $c>0$, $\a\in(0,1]$, $\be\in(0,1]$ and $\ga\in(0,1]$ are
constants. It can be verified that the corresponding spectral
measure is continuous in space $x$ and discrete in time $t$.
See Ma (2003a, 2003b) for more examples of stationary
covariance models.

In the following, we verify that the sample functions of these
space-time models are fractals. We will check Conditions
(D1) and (D2) first, and then obtain the corresponding
Hausdorff dimension results from Theorems \ref{Th:dim}
and \ref{Th:dim2}.

\begin{proposition}\label{Le:Cre1}
Let $X=\{X(x,t), (x,t)\in\d\times\R\}$ be a centered
stationary Gaussian random field in $\R$ with covariance function
as (\ref{Eq:Cre}). Then for any $M > 0$,
there exist constants $c_{13}>0$ and $c_{14}>0$ such that
\begin{equation}\label{Eq:si2}
c_{13}\(|x-y|^{2\ga}+|t-s|^{2\a}\)\leq\E\(X(x,t)-X(y,s)\)^2 \leq
c_{14}\(|x-y|^{2\ga}+|t-s|^{2\a}\)
\end{equation}
and
\begin{equation}\label{Eq:si3}
{\rm Var}\big(X(x,t)\big|X(y,s)\big) \geq
c_{13}\, \big(|x-y|^{2\ga}+|t-s|^{2\a}\big)
\end{equation}
for all $(x,t)$ and $(y,s) \in [-M, M]^{d+1}$.
\end{proposition}

\begin{proposition}\label{Pr:Credim}
Let $X=\{X(x,t), (x,t)\in\d\times\R\}$ be a centered
stationary Gaussian random field in $\R$ with covariance function
as (\ref{Eq:Cre}), and let $\X$ be its associated $(N, p)$-random
field defined by (\ref{def:X}). Then, with probability 1,
\begin{equation}\label{Eq:Credimr}
\dim \,\X([0,1]^{d+1})=\min\bigg\{p;
\,\frac{d}{\ga}+\frac{1}{\a}\bigg\}.
\end{equation}
And if $0<\a\leq\ga<1$, then

\begin{equation}\label{Eq:CreGr1}
\dim \,{\rm Gr}\X([0,1]^{d+1})=\left\{
\begin{array}{ll} d+1+(1-\a)p \qquad
 &{\rm if}\,\, p < \frac{1}{\a},\\
 d+\frac{\ga}{\a}+(1-\ga)p
 &{\rm if}\,\, \frac{1}{\a}\leq p < \frac{1}{\a}+\frac{d}{\ga},\\
 \frac{1}{\a}+\frac{d}{\ga}
 &{\rm if}\,\, p \geq \frac{1}{\a}+\frac{d}{\ga}.\\
\end{array}
\right.
\end{equation}
If $0<\ga\leq\a<1$, then
\begin{equation}\label{Eq:CreGr2}
\dim \,{\rm Gr}\X([0,1]^{d+1})=\left\{
\begin{array}{ll} d+1+(1-\ga)p \qquad
 &{\rm if}\,\, p < \frac{d}{\ga},\\
 \frac{d\a}{\ga}+1+(1-\a)p
 &{\rm if}\,\, \frac{d}{\ga}\leq p < \frac{1}{\a}+\frac{d}{\ga},\\
 \frac{1}{\a}+\frac{d}{\ga}
 &{\rm if}\,\, p \geq \frac{1}{\a}+\frac{d}{\ga}.\\
\end{array}
\right.
\end{equation}
\end{proposition}

\begin{proposition}\label{Pr:Credim2}
Let $X=\{X(x,t), (x,t)\in\d\times\R\}$ be a centered
stationary Gaussian random field in $\R$ with covariance function
as (\ref{Eq:Cre}), and let $\X$ be its associated $(N, p)$-random
field.\\
{\rm (i)} When $\frac{1}{\a}+\frac{d}{\ga}<p$, then for
every $x\in\p$, $\X^{-1}(x)=\emptyset$ a.s.\\
{\rm (ii)} When $\frac{1}{\a}+\frac{d}{\ga}>p$,  if
$0<\a\leq \ga \le 1$, then for any $x\in\p$, with
positive probability
\begin{equation}\label{Eq:Crel}
\dim\, \X^{-1}(x)=\left\{
\begin{array}{ll} d+1-\a p \qquad
 &{\rm if}\,\, p < \frac{1}{\a},\\
 d+\frac{\ga}{\a}-\ga p
 &{\rm if}\,\, p\geq \frac{1}{\a}, \\
\end{array}
\right.
\end{equation}
and if $0<\ga\leq\a \le 1$, then for any $x\in\p$, with
positive probability
\begin{equation}\label{Eq:Crel2}
\dim \,\X^{-1}(x)=\left\{
\begin{array}{ll} d+1-\ga p \qquad
 &{\rm if}\,\, p < \frac{d}{\ga},\\
 \frac{d\a}{\ga}+1-\a p
 &{\rm if}\,\, p\geq \frac{d}{\ga}.\\
\end{array}
\right.
\end{equation}
\end{proposition}

\subsection{Stationary spectral density models}
In Section 6.1, the stationary space-time models are constructed
directly by covariance functions, which are isotropic in the space
variable. Stein (2005) showed that stationary covariance functions
which are anisotropic in space can be constructed by choosing
spectral densities of the form
\begin{equation}\label{Eq:st}
f(\la)=\Bigg(\sum_{j=1}^{d+1}c_j\(a_j+|\la_j|^2\)^{\a_j}\Bigg)^{-\nu},
\quad\quad \forall \la\in \d\times\R,
\end{equation}
where $\nu>0$, $c_j>0$, $a_j>0$ and $\a_j\in\mathbb{N}$ for
$j=1,\cdots, d+1$ are constants such that
    \[\sum_{j=1}^{d+1}\frac{1}{\a_j}<2\nu.\]
This last condition guarantees $f\in L^1(\R^{d+1})$. Clearly
$f(\lambda)$ in (\ref{Eq:st}) satisfies (\ref{Eq:Maternb})
with $\beta_j = \alpha_j$ and $\gamma = 2 \nu$.
Hence we may apply our results to analyze this class of models,
through the smoothness properties and the fractal properties.

\begin{proposition}
Let $X=\{X(x,t), (x,t)\in\d\times\R\}$ be a centered
stationary Gaussian random field in $\R$ with spectral
density as (\ref{Eq:st}).\\
{\rm (i)} If
\[
2\nu >\sum_{j=1}^{d+1}\frac{1}{\a_j}+
\frac{2}{\min_{1\leq\ell\leq d+1}\a_\ell},
\]
then $X(x, t)$ is mean square differentiable and has a version 
$\widetilde{X}(x,t)$ which is sample path differentiable 
almost surely.\\
{\rm (ii)} $X$ is a fractal {\rm [}i.e. the sample path of $X$ 
may have fractional Hausdorff dimension{\rm ]} if and only if
\[
\sum_{j=1}^{d+1}\frac{1}{\a_j}< 2\nu \leq
\sum_{j=1}^{d+1}\frac{1}{\a_j}+\frac{2}{\min_{1\leq\ell\leq
d+1}\a_\ell}.
\]
\end{proposition}

The Hausdorff dimensions of various fractals generated by this
kind of models can also be computed using Corollary \ref{Co:SI-dim},
with $H_j=\a_j\(\nu-\sum_{\ell=1}^{d+1}\frac{1}{2\a_\ell}\)$,
$\overline{H}_j=1\wedge H_j$ for $j=1,\cdots, d+1$. We leave the
details to an interested reader.

\section{Proofs}

{\large \textbf{Proof of Proposition \ref{Le:int}}}

Since (\ref{Eq:int}) is equivalent to $\int_{\N}\(1\wedge
|\la|^2\)f(\la)d\la <\infty$, and $\int_{|\la|\leq 1}
|\la|^2f(\la)d\la<\infty$ is given, it is enough for us to show
\[
\int_{|\la|>1}\frac{d\la}{\(\sum_{j=1}^N|\la_j|^{\be_j}\)^{\ga}}
<\infty
\]
is equivalent to (\ref{Eq:leg}).

For this purpose, we appeal to the following fact:
Given positive constants $\be$ and $\gamma$, there exists
a finite constant $c_{15}$ such that for all $a > 0$,
\begin{equation}\label{Eq:Elem}
\begin{split}
\int_0^\infty\frac{d x}{(a+x^{\be})^{\ga}}
&= \left\{\begin{array}{ll}
 c_{15}\, a^{-(\ga-\frac{1}{\be})} &\hbox{if } \, \be \ga > 1,\\
 +\infty &\hbox{if } \, \be \ga \le 1.\\
 \end{array}
 \right.
\end{split}
\end{equation}
To verify this, we make a change of
variable  $x=a^{\frac{1}{\be}}y$ to obtain
\begin{equation*}
\int_0^\infty\frac{d x}{(a+x^{\be})^{\ga}}
= a^{-(\ga-\frac{1}{\be})}\int_0^\infty\frac{d y}{(1+y^\be)^\ga}.
\end{equation*}
Thus (\ref{Eq:Elem}) follows.

First we assume (\ref{Eq:leg}) holds. Since $|\lambda| > 1$
implies that $|\lambda_{j_0}| > \frac 1 {\sqrt{N}}$
for some $j_0 \in \{1, \ldots, N\}$. Without loss of
generality we assume $j_0 = 1$. Then by using (\ref{Eq:Elem})
$(N-1)$ times we obtain
\begin{equation*}
\begin{split}
\int_{|\la|>1}\frac{d\la}{\(\sum_{j=1}^N|\la_j|^{\be_j}\)^{\ga}}
&\leq 2^N\, \int_{\frac{1}{\sqrt{N}}}^\infty d\la_1 \underbrace{
\int_0^\infty \cdots \int_0^\infty}_{N-2}\frac{ d\la_2\cdots d\la_{N-1}}
{\(\sum_{j=1}^{N-1}|\la_j|^{\be_j}\)^{\ga-\frac{1}{\be_N}}}\\
&\leq c\,\int_{\frac{1}{\sqrt{N}}}^\infty\frac{d\la_1}
{\(|\la_1|^{\be_1}\)^{\ga-\sum_{j=2}^N\frac{1}{\be_j}}}<\infty,
\end{split}
\end{equation*}
because $\be_1\(\ga-\sum_{j=2}^N\frac{1}{\be_j}\)>1$.
This proves the sufficiency of (\ref{Eq:leg}).

To prove the converse, we assume (\ref{Eq:leg}) does not hold.
Then there is a unique integer $\tau \in\{1, \ldots, N\}$ such
that $\sum_{i=1}^{\tau -1} \frac 1 {\beta_i}< \gamma \le
\sum_{i=1}^{\tau} \frac 1 {\beta_i}.$ Note that
\[
\int_{|\la|>1}\frac{d\la}{\(\sum_{j=1}^N|\la_j|^{\be_j}\)^{\ga}}
\geq \underbrace{\int_{0}^\infty  \cdots \int_{0}^\infty}_{N-1}
\int_{1}^\infty  \frac{d\la_1\cdots d\la_{N}}
{\(\sum_{j=1}^{N}|\la_j|^{\be_j}\)^{\ga}}.
\]
By using (\ref{Eq:Elem}) and integrating $d\la_1\cdots d\la_{\tau}$,
we see that the last integral is divergent. This finishes
the proof. {\qed}

\vspace{.1in}
\noindent{\large \textbf{Proof of Lemma \ref{Le:St}}}

For any $s,t\in \R^N$,  denote
$\hat{s}_0=t$, $\hat{s}_1=(s_1, t_2,\cdots,t_N)$,
$\hat{s}_2=(s_1,s_2,t_3,\cdots,t_N)$, $\cdots$,
$\hat{s}_{N-1}=(s_1,\cdots,s_{N-1},t_N)$ and
$\hat{s}_N=s$. Let $h=s-t\triangleq(h_1,\cdots,h_N)$.
By Jensen's inequality, (\ref{Eq:sigmaa}) and (\ref{Eq:aniMat})
we can write
\begin{equation}\label{Eq:si2St}
\begin{split}
\E\(X(s)-X(t)\)^2 
&\leq N\,\sum_{k=1}^N \E\(X(\hat{s}_k)-X(\hat{s}_{k-1})\)^2\\
&=2N\,\sum_{k=1}^N\int_{\N}\(1-\cos (h_k\la_k)\)f(\la)d\la\\
&\le 2N\,\sum_{k=1}^N\int_{|\la|\leq1}\(1-\cos (h_k\la_k)\)f(\la)d\la \\
&\qquad\qquad  +2Nc_5\, \sum_{k=1}^N\int_{|\la|>1}
\(1-\cos (h_k\la_k)\)\frac{d\la}
{\(\sum_{i=1}^N|\la_i|^{H_i}\)^{Q+2}}\\
&\triangleq I_1+I_2.
\end{split}
\end{equation}
By using the inequality $1 - \cos x \le x^2$ we have
\begin{equation}\label{Eq:I1b}
I_1\leq 2N\bigg(\sum_{k=1}^N h_k^2\bigg)\,
\int_{|\la|\leq 1}|\la|^2 f(\la)d\la \leq c_{16}\, |s-t|^2
\end{equation}
for some positive and finite constant $c_{16}$.

To bound the $k$th integral in $I_2$, we note that, when $|\lambda| > 1$,
either $|\lambda_k |> \frac{1}{\sqrt{N}}$ or there is $j_0 \ne k$
such that $|\lambda_{j_0} |> \frac{1}{\sqrt{N}}$. We break the integral
according to these two possibilities.
\begin{equation}\label{Eq:I2b}
\begin{split}
&\int_{|\la|>1}\(1-\cos (h_k\la_k)\)\frac{d\la}
{\(\sum_{i=1}^N|\la_i|^{H_i}\)^{Q+2}}\\
&\leq 2\int_{\frac{1}{\sqrt{N}}}^\infty \(1-\cos (h_k\la_k)\) d\la_k
\int_{\R^{N-1}}\frac{d\la_1\cdots d\la_{k-1}
d\la_{k+1}\cdots d\la_{N}}{\(\sum_{i=1}^{N}|\la_i|^{H_i}\)^{Q+2}} \\
&\quad\qquad + 4 \int_0^1 \(1-\cos (h_k\la_k)\)d\la_k \int_{\frac{1}
{\sqrt{N}}}^\infty d\la_{j_0}\int_{\R^{N-2}}  \frac{d\la_{k, j_0}^{\vee}}
{\(\sum_{i=1}^{N}|\la_i|^{H_i}\)^{Q+2}}\\
&\triangleq I_3+I_4,
\end{split}
\end{equation}
where $d\la_{k, j_0}^{\vee}$ denotes integration in $\la_i$ ($i \ne k, j_0$).

By using (\ref{Eq:Elem}) repeatedly [$N-1$ times], we obtain
\begin{equation}\label{Eq:I2c}
\begin{split}
I_3& \le c\, \int_{\frac{1}{\sqrt{N}}}^\infty \frac{1-\cos (h_k\la_k)}
{|\la_k|^{2H_k+1}}\, d\la_k\\
&\le c\, \Bigg(\int_{\frac{1}{\sqrt{N}}}^{\frac{1}{|h_k|}}
\frac{h_k^2\la_k^2}{\la_k^{2H_k+1}}d\la_k+
\int_{\frac{1}{|h_k|}}^\infty \frac{1}{\la_k^{2H_k+1}}d\la_k\Bigg)\\
&\le c\, \sigma_k(|h_k|),
\end{split}
\end{equation}
where $\sigma_k$ is defined as in (\ref{Def:sigma}).

Similarly, we use (\ref{Eq:Elem}) $N-2$ times to get
\begin{equation}
\begin{split}\label{Eq:I4b}
I_4&\leq c\, \int_0^1 \(1-\cos (h_k\la_k)\)d\la_k
\int_{\frac{1}{\sqrt{N}}}^\infty \frac{d\la_{j_0}}{\(\la_k^{H_k}
+ \la_{j_0}^{H_{j_0}}\)^{2+\frac{1}{H_k}+\frac{1}{H_{j_0}}}}\\
&\le c\, \int_0^1 \(1-\cos (h_k\la_k)\)d\la_k \int_{\frac{1}
{\sqrt{N}}}^\infty \frac{d\la_{j_0}}{ \la_{j_0}^{2H_{j_0}
+1+\frac{H_{j_0}}{H_k}}}\\
&\leq c\, |h_k|^2.
\end{split}
\end{equation}
Combining (\ref{Eq:si2St})--(\ref{Eq:I4b}) yields the
upper bound in (\ref{Eq:sigmab}).

Next we prove the lower bound in (\ref{Eq:sigmab}).
By (\ref{Eq:sigmaa})  and (\ref{Eq:aniMat}) we have
\begin{equation}
\E\(X(s)-X(t)\)^2  \ge c_4\, \int_{|\la|>1}\(1-\cos\l s-t,\la\r\)\,
\frac{d\la}{\rho(\la)^{Q+2}},
\end{equation}
where $\rho(\la)=\sum_{j=1}^N|\la_j|^{H_j},\,\la\in\N$.
So, for the lower bound of $\E\(X(s)-X(t)\)^2$, it is enough to
show that for every $j=1,\cdots,N$ and all
$h \in \R^N$, we have
\begin{equation}\label{Eq:lower}
\int_{|\la|>1}\(1-\cos\l h,\la\r\)\,
\frac{d\la}{\rho(\la)^{Q+2}}\geq c \sigma_j(|h_j|),
\end{equation}
where $c$ is a positive constant.

We only prove (\ref{Eq:lower}) for $j =1$, and the other
cases are similar. Fix $h \in \R^N$
with $|h_1| > 0$ [otherwise there is nothing to prove]
and we make a change of variables
\[
y_{\ell}=\rho(h)^{H_{\ell}^{-1}}\la_{\ell}, \qquad
\forall \ell=1,\cdots,N.
\]
We consider a subset
of the integration region defined by
$$
D(h)= \Big\{y\in \R^N: |y_1|\in [\rho(h)^{H_1^{-1}}, 1],\
|y_\ell| \le 1\ \hbox{ and } y_\ell h_\ell > 0
\hbox{ for } 1\le \ell \le N\Big\}.
$$
Since $\rho(\la)=\rho(y)/\rho(h)$, we have
\begin{equation}\label{Eq:low}
\begin{split}
&\int_{|\la|>1}\(1-\cos\l h,\la\r\)\frac{d\la}{\rho(\la)^{Q+2}}
\ge \rho(h)^2\, \int_{D(h)} \frac{1-\cos\Big(\sum_{\ell=1}^N
h_{\ell}\rho(h)^{-H_{\ell}^{-1}}y_{\ell}\Big)}
{\(\sum_{\ell=1}^N|y_\ell|^{H_\ell}\)^{Q+2}}\, d y.
\end{split}
\end{equation}
By using the inequality $1- \cos x\ge c\, x^2$
for all $|x|\le N$, where $c>0$ is a constant, and the fact that
$h_\ell y_\ell > 0$ for all $1 \le \ell \le N$, we derive that the
last integral is at least [up to a constant]
\begin{equation}\label{Eq:low2}
\begin{split}
&\rho(h)^2\int_{D(h)} \frac{\Big(\sum_{\ell=1}^N
h_{\ell}\rho(h)^{-H_{\ell}^{-1}}y_{\ell}\Big)^2}
{\(\sum_{\ell=1}^N|y_\ell|^{H_\ell}\)^{Q+2}}\, d y \\
&\geq \rho(h)^2\int_{\rho(h)^{H_1^{-1}}}^1 h_1^2
\rho(h)^{-\frac{2}{H_1}}y_1^2d y_1
\underbrace{\int_0^1\cdots\int_0^1}_{N-1}
\frac{d y_2\cdots d y_N}
{\(\sum_{\ell=1}^N|y_\ell|^{H_\ell}\)^{Q+2}}\\
&\geq c\, \rho(h)^{2-\frac{2}{H_1}}\, h_1^2\,
\int_{\rho(h)^{H_1^{-1}}}^1\frac{y_1^2 d y_1}
{\(y_1^{H_1}\)^{\frac{1}{H_1}+2}}\\
&= c\,  \sigma_1(|h_1|).
\end{split}
 \end{equation}
This proves (\ref{Eq:lower}) and hence Lemma \ref{Le:St}.
{\qed}

In order to prove Theorem \ref{Th:pred-error}, we will make use of
the following lemma which implies that the prediction error of $X$
is determined by the behavior of the spectral density $f(\la)$ at
infinity.
\begin{lemma}\label{Le:main}
Assume (\ref{Eq:Maternb}) is satisfied, then for any fixed
constant $T>0$, there exists a positive and finite
constant $c_{17}$ such that for all functions $g$ of the form
\[
g(\la)=\sum_{k=1}^n a_k\(e^{i\l t^k,\la\r}-1\),
\]
where $a_k\in\R$ and $t^k\in [-T,T]^N$, we have
\begin{equation}
|g(\la)|\leq c_{17}\,|\la|\bigg(\int_{\N}|g(\xi)|^2f(\xi)d\xi\bigg)^{1/2}
\end{equation}
for all $\la \in \R^N$ that satisfy $|\la|\le 1$.
\end{lemma}

\begin{proof}\ By (\ref{Eq:Maternb}), we can find positive
constants $C$ and $\eta$, such that
\[
f(\la)\geq\frac{C}{|\la|^\eta},
\quad\forall\la \in \N  \text{ with } |\la|\text{ large enough}.
\]
Then the desired result follows from the proof of Lemma 2.2
in Xiao (2007).
\end{proof}

\vspace{.1in} \noindent{\large \textbf{Proof of Theorem
\ref{Th:pred-error}}}

First, let's prove the upper bound in (\ref{Eq:C3'}).
By Lemma \ref{Le:St} we have
\begin{equation}\label{Eq:pred1}
\begin{split}
\Var\Big(X(u)|X(t^1),\cdots,X(t^n)\Big)
&\leq\min_{0\leq k\leq n}\E\big(X(u)-X(t^k)\big)^2\\
&\leq c_9\min_{0\leq k\leq n}\sum_{j=1}^N
\sigma_j\big(|u_j-t^k_j|\big).
\end{split}
\end{equation}

In order to prove the lower bound for the conditional
variance in (\ref{Eq:C3'}), we denote
$r\equiv \min\limits_{0\leq k\leq
n}\sum_{j=1}^N|u_j-t^k_j|^{H_j}$. Working in the Hilbert space
setting, the conditional variance is just the square of 
$L^2(\P)$-distance of $X(u)$ from the subspace generated by
$\{X(t^1),\cdots,X(t^n)\}$, so it is sufficient to
prove that for all $a_k\in \R$, $1\leq k\leq n$,\\
\begin{equation}\label{Eq:ms}
\E\bigg(X(u)-\sum_{k=1}^n a_kX(t^k)\bigg)^2\geq c_6\, r^2,
\end{equation}
where $c_6$ is a positive constant which may only depend on
$H_1, \ldots, H_N$ and $N$.

By using the stochastic integral representation (\ref{Eq:eks})
of $X$, the left hand side of (\ref{Eq:ms}) can be written as
\begin{equation}\label{Eq:mse}
\E\bigg(X(u)-\sum_{k=1}^n a_kX(t^k)\bigg)^2=\int_{\N}\bigg|e^{i\l
u,\la\r}-1-\sum_{k=1}^n a_k(e^{i\l t^k,\la\r}-1)\bigg|^2f(\la)\,
d\la.
\end{equation}
Hence, we only need to show
\begin{equation}\label{Eq:msel}
\int_{\N}\bigg|e^{i\l u,\la\r}-\sum_{k=0}^n a_k e^{i\l
t^k,\la\r}\bigg|^2f(\la)\,d\la\geq c_6\, r^2,
\end{equation}
where $t^0=0$ and $a_0=1-\sum_{k=1}^n a_k$.
\medskip

We choose a function $\de(\cdot): \N\to[0,1]$ in $C^\infty(\N)$ [the 
space of all infinitely differentiable functions defined on $\R^N$] 
such that $\de(0)=1$ and it vanishes outside the open set
$\big\{t\in\N:\,\, \sum_{j=1}^N |t_j|^{H_j}<1\big\}$. Denote by
$\hat{\de}$ the Fourier transform of $\de$. Then one can verify that
$\hat{\de}(\cdot)\in C^\infty(\N)$ as well and $\hat{\de}(\la)$
decays rapidly as $|\la|\to\infty$.\\

Let $E$ be the $N\times N$ diagonal matrix with $H_1^{-1},
\cdots,H_N^{-1}$ on its diagonal and let $\de_r(t)=r^{-Q}\,
\de(r^{-E}t)$ for all $t \in \R^N$. Then the
inverse Fourier transformation and a change of variables yield
\begin{equation}\label{Eq:Fourier}
\de_r(t)=(2\pi)^{-N}\int_{\N}e^{-i\l
t,\la\r}\hat{\de}(r^E\la)\,d\la.
\end{equation}
Since $\min\big\{\sum_{j=1}^N |u_j-t^k_j|^{H_j}: 0\leq k\leq
n\big\}\geq r$, we have $\de_r(u-t^k)=0$ for $k=0,1,\cdots,n$.
This and (\ref{Eq:Fourier}) together imply that
\begin{equation}\label{Eq:J}
\begin{split}
J:&=\int_{\N}\Bigg(e^{i\l u,\la\r}-\sum_{k=0}^n a_k e^{i\l
t^k,\la\r}\Bigg)e^{-i\l u,\la\r}\hat{\de}(r^E\la)\,d\la\\
&=(2\pi)^N\Bigg(\de_r(0)-\sum_{k=0}^n a_k\de_r(u-t^k)\Bigg)\\
&=(2\pi)^N r^{-Q}.
\end{split}
\end{equation}

Now we split the integral in (\ref{Eq:J}) over $\{\la: |\la|<1\}$
and $\{\la: |\la|\geq 1\}$ and denote the two integrals
by $I_1$ and $I_2$, respectively. It follows from Lemma
\ref{Le:main} that
\begin{equation}\label{Eq:I1}
\begin{split}
I_1&\leq\int_{|\la|<1}\Big|e^{i\l u,\la\r}-\sum_{k=0}^n a_k
e^{i\l t^k,\la\r}\Big||\hat{\de}(r^E\la)|d\la\\
&\leq c_{17}\, \Bigg[\int_{\N}\Big|e^{i\l u,\la\r}-\sum_{k=0}^n a_k
e^{i\l t^k,\la\r}\Big|^2 \,f(\la)\,d\la\Bigg]^{1/2}
\int_{|\la|<1}|\la||\hat{\de}(r^E\la)|d\la\\
&\leq c_{18}\,\Bigg[\E\bigg(X(u)-\sum_{k=1}^n
a_kX(t^k)\bigg)^2\Bigg]^{1/2},
\end{split}
\end{equation}
where the last inequality follows from (\ref{Eq:mse}) and the
boundedness of $\hat{\de}$.

On the other hand, by the Cauchy-Schwarz inequality and
(\ref{Eq:mse}), we have
\begin{equation}\label{Eq:I2}
\begin{split}
I_2^2&\leq\int_{|\la|\geq 1}\bigg|e^{i\l u,\la\r}-\sum_{k=0}^n
a_k e^{i\l t^k,\la\r}\bigg|^2f(\la)d\la\int_{|\la|\geq 1}
\frac{1}{f(\la)}|\hat{\de}(r^E\la)|^2\,d\la\\
&\leq\E\bigg(X(u)-\sum_{k=1}^n a_kX(t^k)\bigg)^2r^{-Q}
\int_{|\la|\geq 1}\frac{1}{f(r^{-E}\la)}|\hat{\de}(\la)|^2\,d\la\\
&=\E\bigg(X(u)-\sum_{k=1}^n a_kX(t^k)\bigg)^2r^{-2Q-2}
\int_{|\la|\geq 1}\frac{1}{f(\la)}|\hat{\de}(\la)|^2\,d\la.\\
\end{split}
\end{equation}
The last integral is convergent thanks to the fast decay of
$\hat{\de}(\la)$. Finally, combining (\ref{Eq:J}), (\ref{Eq:I1})
and (\ref{Eq:I2}), we get
\[
(2\pi)^N r^{-Q}\leq c_{19}\,\Bigg[\E\bigg(X(u)-\sum_{k=1}^n
a_kX(t^k)\bigg)^2\Bigg]^{1/2}r^{-Q-1}.
\]
Henceforth (\ref{Eq:ms}) follows, and the theorem was proved
because of (\ref{Eq:pred1}) and (\ref{Eq:ms}). {\qed}



\vspace{.1in}
\noindent{\large \textbf{Proof of Theorem \ref{Th:diff}}}

For $t\in\N$, it is known that $X_{j,h}=\frac{X(t+he_j)-X(t)}
{h}$ converges in $L^2$, as $h\to 0$, if and only if
\[D_{h,k}\triangleq \frac{1}{h k}\E\Big\{\(X(t+he_j)-X(t)\)
\(X(t+k e_j)-X(t)\)\Big\}\]
converges to a constant as $h,k\to 0$. However,
\begin{equation}\label{Eq:diffs}
\begin{split}
D_{h,k}&=\frac{1}{h k}\Big\{C(t+he_j,t+k e_j)-C(t,t+k e_j)
-C(t+he_j,t)+C(t,t)\Big\}\\
      &=\frac{1}{2h k}\Big\{v(he_j)+v(k e_j)-v\((h-k)e_j\)\Big\}.
     \end{split}
    \end{equation}
So the first part of the theorem is proved.
For the second part, it is clear that if $v(t)$ has
second-order partial derivatives at 0 in the $j$-th
direction then (\ref{Eq:msd-nsc}) holds [thanks to Taylor's
theorem]. On the other hand, if (\ref{Eq:msd-nsc}) holds,
then by taking $h=k \to 0$ in (\ref{Eq:diffs}) we see that
$\partial v/\partial t_j(0) = 0$. This fact, together with
(\ref{Eq:msd-nsc}), implies that
\begin{equation*}
\begin{split}
\frac{\partial^2 v} {\partial t_j^2}(0)
&= \lim_{k\to 0} \frac 1 k \lim_{h \to 0} \frac{v((k+h)e_j) - v(ke_j)} h\\
&= \lim_{k\to 0} \lim_{h \to 0} \frac{v((k+h)e_j) - v(ke_j) + v(he_j)} {hk}
\end{split}
\end{equation*}
exists. This finishes the proof of Theorem \ref{Th:diff}. {\qed}

\vspace{.1in} \noindent{\large \textbf{Proof of Corollary
\ref{Th:MSdiff}}}

By Theorem \ref{Th:diff} it amounts to show that
$\lim\limits_{h, k \to 0}D_{h,k}$ exists if and
only if (\ref{Eq:mqdiff}) holds. [i.e., $\be_j\(\ga-
\sum_{i=1}^N\frac{1}{\be_i}\)>2$.]
It follows from (\ref{Eq:diffs}) and (\ref{Eq:sigmaa}) that
\begin{equation}\label{Eq:diffs2}
D_{h,k}=\int_{\N}\frac{1-\cos\l he_j,\la\r -
\cos\l k e_j,\la\r + \cos\l(h-k)e_j,\la\r} {h k}\, f(\la)\,d\la.
\end{equation}
To prove the sufficiency of (\ref{Eq:mqdiff}), we note that for
each fixed $\la\in\N$,
\begin{equation}\label{Eq:diffs2b}
\lim_{h,k\to 0}\frac{1-\cos(h\la_j) -
\cos (k \la_j) + \cos((h-k)\la_j)}
{h k}=\la_j^2
\end{equation}
and by the mean value theorem,
\[
\bigg|\frac{1-\cos(h\la_j) -
\cos (k \la_j) + \cos((h-k)\la_j)}
{h k}\bigg| \le \lambda_j^2.
\]

Now we assume (\ref{Eq:mqdiff})
holds. Then, as in the
proof of Proposition \ref{Le:int}, we have
\[
\int_{\la \in \R^N: |\la_j|>1}\frac{\la_j^2\,d\la}
{\(\sum_{i=1}^N|\la_i|^{\be_i}\)^{\ga}}
\le c \int_1^\infty \frac{\la_j^2\, d\la_j}
 {\la_j^{\beta_j (\ga - \sum_{i\ne j}\frac{1}{\be_i})}}<\infty.
\]
This implies $ \int_{\N}\la_j^2f(\la)d\la < \infty.$
By (\ref{Eq:diffs2}), (\ref{Eq:diffs2b}) and the
dominated convergence theorem, we obtain
\[
\lim_{h,k\to 0}D_{h,k}=\int_{\N}\la_j^2f(\la)d\la.
\]

To prove the necessity of (\ref{Eq:mqdiff}), we assume
$
\be_j\big(\ga-\sum_{i=1}^N\frac{1}{\be_i}\big)\le2.
$
Then, as in the proof of Proposition \ref{Le:int}, we have
\begin{equation}\label{Eq:div}
\int_{\la \in \R^N: |\la_j|>1}\frac{\la_j^2\,d\la}
{\(\sum_{i=1}^N|\la_i|^{\be_i}\)^{\ga}}
=\infty.
\end{equation}
We let $h=k \downarrow 0$ and use Fatou's lemma to (\ref{Eq:diffs2})
[note the integrand is non-negative] to derive
\[
\liminf_{h=k\downarrow 0}D_{h,k} \ge \int_{\R^N} \la_j^2f(\la)d\la
= \infty,
\]
where the last equality follows from (\ref{Eq:div}). So
$\lim\limits_{h, k \to 0}D_{h,k}$ does not exist and the proof
is finished.  {\qed}

\vspace{.1in}
\noindent{\large
\textbf{Proof of Theorem \ref{Th:msdifferentiability}}}

If $v(t)$ has continuous second-order partial derivatives,
then Theorem \ref{Th:diff} implies that $X$ has mean square
partial derivatives in all $N$ directions. Let $\nabla_X(t)
=\(X'_1(t),\cdots,X'_N(t)\)^T$ and we show that it satisfies
(\ref{Eq:msd}).

For any unit vector $u$ in $\N$, we can write it as
$u=\sum_{j=1}^N u_j e_j$ and $\sum_{j=1}^N u_j^2=1$. So
$u^T\nabla_X(t)=\sum_{j=1}^N u_j X'_j(t)$. Hence
\begin{equation}\label{Eq:totdev}
\begin{split}
&\E\Bigg(\frac{X(t+hu)-X(t)}{h}-u^T\nabla_X(t)\Bigg)^2\\
&=\E\Bigg(\frac{X(t+h u)-X(t)}{h}-\sum_{j=1}^N u_j X'_j(t)\Bigg)^2\\
&=\frac{1}{h^2}v(h u)+\E\bigg(\sum_{j=1}^N u_j X'_j(t)\bigg)^2-
\frac{2}{h}\sum_{j=1}^N u_j\Big(\E X(t+h u)X'_j(t)-\E X(t)X'_j(t)\Big)\\
&=\frac{1}{h^2}v(h u)+\E\bigg(\sum_{j=1}^N u_j X'_j(t)\bigg)^2
-\frac{1}{h}\sum_{j=1}^N u_j v'_j(h u).\\
\end{split}
\end{equation}
The last equality in (\ref{Eq:totdev}) follows from (\ref{Eq:crosscov}).

Since $v(t)$ is an even function with $v(0)=0$ and has
continuous second-order partial and mixed partial derivatives,
then Taylor's theorem implies
\begin{equation}\label{Eq:1stterm}
\begin{split}
\lim_{h\to 0}\frac{1}{h^2}v(h u)
=\lim_{h\to 0}\frac{v(h u)+v(-h u)-2 v(0)}{2h^2}
=\frac{1}{2} u^T\Omega(0)u,
\end{split}
\end{equation}
where $\Omega(0)$ is an $N\times N$ matrix, with
$\(\Omega(0)\)_{i j}=v_{ij}^{''}(0)$ for $i\neq j$,
and $\(\Omega(0)\)_{i i}=v_{i}^{''}(0)$. 	

For the second term in the last line of (\ref{Eq:totdev}),
note that for any $i$, $j=1,\cdots,N$ and $i\neq j$,
and any $l>0$, $m>0$,
\begin{equation}\label{Eq:crosscovij}
\begin{split}
&\E\Bigg(\frac{X(t+l e_i)-X(t)}{l}\frac{X(t+m e_j)-X(t)}{m}\Bigg)\\
=&\frac{1}{l m}\E\Big(X(t+l e_i)X(t+m e_j)-X(t)X(t+m e_j)
-X(t+l e_i)X(t)+X^2(t)\Big)\\
=&\frac{1}{2l m}\Big(v(l e_i)+v(-m e_j)-v(l e_i-m e_j)\Big).\\
\end{split}
\end{equation}
Let $l\to 0$, $m\to 0$, then the last term in (\ref{Eq:crosscovij}) goes
to $\frac{1}{2}v_{ij}^{''}(0)$, where $v_{ij}^{''}(0)$ is
the second-order mixed partial derivative of $v$ at 0 in
the $i$-th and $j$-th directions. By Theorem \ref{Th:deriv},
we have $\E\(X'_j(t)\)^2=\frac{1}{2}v_j^{''}(0)$, for
$j=1,\cdots,N$. Hence 				
\[
\E\bigg(\sum_{j=1}^N u_j X'_j(t)\bigg)^2=\frac{1}{2}u^T\Omega(0)u.
\]

Finally for the last term in (\ref{Eq:totdev}), we use
Taylor's theorem again to derive
\[
\lim_{h\to 0}\frac{1}{h}\sum_{j=1}^N u_j v'_j(h u)
= u^T\Omega(0)u.
\]
Combining this with (\ref{Eq:1stterm}) and (\ref{Eq:crosscovij})
shows that (\ref{Eq:totdev}) goes to 0, as $h\to 0$.
This finishes the proof. {\qed}
	
\vspace{.1in}
\noindent{\large \textbf{Proof of Theorem
\ref{Th:spdiff}}}

Under (\ref{Eq:mqdiff}), Corollary \ref{Th:MSdiff} ensures that
the mean square partial derivative
$X'_j(t)$ exists. In order to show that $X'_j(t)$ has a continuous
version, by Kolmorogov's continuity theorem or general
Gaussian theory [cf. Adler (1981), Adler and Taylor (2007)],
it is enough to show there exist
constants $c_{20}>0$ and $\eta>0$ such that
\begin{equation}\label{Eq:moment}
\E\big[X'_j(s)-X'_j(t)\big]^2\leq c_{20}\,|s-t|^{\eta}.
\end{equation}
Recall that
\begin{equation*}
\begin{split}
C(s,t)&=\int_{\N}\(e^{i\l s,\la\r}-1\)\(e^{-i\l t,\la\r}-1\)f(\la)d\la\\
&=\int_{\N}\Big[\cos\l s-t,\la\r-\cos\l t,\la\r-\cos\l s,\la\r+1\Big]f(\la)d\la.
\end{split}
\end{equation*}
Thanks to (\ref{Eq:mqdiff}), we derive
\[
\frac{\partial C(s,t)}{\partial s_j}
=\int_{\N}\Big[-\la_j\sin\l s-t,\la\r
+\la_j\sin\l s,\la\r\Big]f(\la)\,d\la
\]
and
\[
 \frac{\partial C(s,t)}{\partial s_j\partial t_j}
 =\int_{\N}\la_j^2\cos\l s-t,\la\r f(\la)\, d\la.
\]
So
\begin{equation*}
\begin{split}
\E\big(X'_j(s)-X'_j(t)\big)^2
&=\E\(X'_j(s)\)^2+\E\(X'_j(t)\)^2-2\E\(X'_j(s)X'_j(t)\)\\
&=2\int_{\N}\la_j^2\(1-\cos\l s-t,\la\r\)f(\la)d\la.
\end{split}
\end{equation*}

The rest of the proof is similar to that of Lemma \ref{Le:St}.
Denote $\hat{s}_0=t$, $\hat{s}_1=(s_1, t_2,\cdots,t_N)$,
$\hat{s}_2 =(s_1,s_2,t_3,\cdots,t_N)$, $\cdots$,
$\hat{s}_{N-1}=(s_1,\cdots,s_{N-1},t_N)$ and $\hat{s}_N=s$.
Then
\begin{equation}\label{Eq:de-inc1}
\begin{split}
\E\(X'_j(s)-X'_j(t)\)^2
&\leq N \sum_{k=1}^N\E\(X'_j(\hat{s}_k)-X'_j(\hat{s}_{k-1})\)^2\\
&=2N\sum_{k=1}^N\Bigg\{\int_{|\la|\leq 1}\la_j^2
\(1-\cos( s_k-t_k)\la_k\)f(\la)d\la\\
&\quad\quad\qquad\qquad\qquad +\int_{|\la|>1}\la_j^2
\(1-\cos(s_k-t_k)\la_k\)f(\la)d\la\Bigg\}\\
&\leq c_{21} \, |s-t|^2+ c_{22}\, \sum_{k=1}^N
\int_{|\la|>1}\frac{\(1-\cos( s_k-t_k)\la_k\)\la_j^2}
{\(\sum_{i=1}^N|\la_i|^{H_i}\)^{Q+2}}\, d\la.
\end{split}
\end{equation}
Now we estimate the last $N$ integrals in (\ref{Eq:de-inc1}).
For simplicity of notation, we only consider the case
when $k=j$ [the cases of $k \ne j$ are similar]. Denote
$h_k=s_k-t_k$ and, similar to (\ref{Eq:I2b}),
(\ref{Eq:I2c}) and (\ref{Eq:I4b}), we derive
\begin{equation*}
\begin{split}
&\int_{|\la|>1}\frac{\(1-\cos(s_k-t_k)\la_k\)\, \la_k^2}
{\(\sum_{i=1}^N|\la_i|^{H_i}\)^{Q+2}}\, d\la\\
&\leq 2 \int_{\frac1 {\sqrt{N}}}^\infty
\(1-\cos( h_k\la_k)\) \lambda_k^2\,d\la_k
\int_{\R^{N-1}}\frac{d\la_1\cdots d\la_{k-1}
d\la_{k+1}\cdots d\la_{N}}
{\(\sum_{i=1}^N|\la_i|^{H_i}\)^{Q+2}} \\
&\qquad + 2 \int_0^1\(1-\cos( h_k\la_k)\) \lambda_k^2\,d\la_k
\int_{\frac1 {\sqrt{N}}}^\infty d\la_{j_0}\int_{\R^{N-2}}
\frac{d\la^{\vee}_{k, j_0}}
{\(\sum_{i=1}^N|\la_i|^{H_i}\)^{Q+2}} \\
&\le c\, \int_{\frac1 {\sqrt{N}}}^\infty
\frac{ \(1-\cos( h_k\la_k)\)\la_k^2} {\lambda_k^{2H_k + 1}}\,d\la_k
+ c\, \int_0^1 \la_k^2 \(1-\cos( h_k\la_k)\) d\la_k \\
&\le c_{23}\Big(|h_k|^{2(H_k- 1)} \log \frac 1 {|h_k|} + |h_k|^2\Big),
\end{split}
\end{equation*}
thanks to $H_k > 1$. Combining this with (\ref{Eq:de-inc1})
proves (\ref{Eq:moment}).
\medskip

It follows from (\ref{Eq:moment})) that the Gaussian field
$X'_j = \{X'_j(t), t \in \R^N\}$ has
a continuous version [which will still be denoted by $X'_j$].
Now we define a new Gaussian random field $\widetilde{X} = \{
\widetilde{X}(t), t \in \R^N\}$ by
\begin{equation}\label{def:X-version1}
\widetilde{X}(t) = X(t_1,\cdots, t_{j-1}, 0,t_{j+1},\cdots, t_N)
+\int_0^{t_j} X'_j(t_1,\cdots, t_{j-1}, s_j,t_{j+1},\cdots, t_N)\,
ds_j.
\end{equation}
Then we can verify that $\widetilde{X}$ is a continuous version
of $X$ and, for every $t \in \R^N$, $\widetilde{X}'_j(t) = X'_j(t)$
almost surely. This amounts to verify that for every $t \in \R^N$,
\[
\E\big(\widetilde{X}(t)^2\big) = v(t)\quad \hbox{ and }\quad
\E\Big[\big(\widetilde{X}(t)- X(t)\big)^2\Big] = 0,
\]
which can be proved  by using (\ref{def:X-version1}), Theorem
\ref{Th:deriv} and (\ref{Eq:crosscov}). Since the verification
is elementary, we omit the details. This proves Part (i) of
Theorem \ref{Th:spdiff}.

It remains to prove Part (ii) of Theorem \ref{Th:spdiff}. By applying
Part (i) to $j=1$, we derive that there is a continuous version
$\widetilde{X^{(1)}}$ of $X$ such that $\frac{\partial
\widetilde{X^{(1)}}} {\partial t_1}(t)$
is continuous. Then we apply Part (i) to $\widetilde{X^{(1)}}$ with $j=2$
and obtain a version $\widetilde{X^{(2)}}$ of $\widetilde{X^{(1)}}$
defined by
\begin{equation}\label{def:X-version2}
\widetilde{X^{(2)}}(t) = \widetilde{X^{(1)}}(t_1,0, t_3,\cdots, t_N)
+\int_0^{t_j} \widetilde{X^{(1)}}'_2(t_1,s_2, t_3,\cdots, t_N)\,
ds_2.
\end{equation}
Then $\frac{\partial \widetilde{X^{(2)}}} {\partial t_1}(t)$
and $\frac{\partial \widetilde{X^{(2)}}} {\partial t_2}(t)$
are almost surely continuous. Repeat this ``updating''
procedure for $j = 3, \cdots, N$, we obtain a continuous
version  $\widetilde{X^{(N)}}$ of $X$ such that all first-order
partial derivatives of $\widetilde{X^{(N)}}$ are continuous
almost surely. Hence the sample function of $\widetilde
{X^{(N)}}$ is almost surely differentiable in the sense
of (\ref{def:sample-diff}). The proof of Theorem
\ref{Th:spdiff} is finished.


\newpage
\noindent{\large \textbf{Proof of Proposition
\ref{Le:Cre1}}}

By stationarity, we'll have
\begin{equation*}
\begin{split}
\E\(X(x,t)-X(y,s)\)^2&=\E\( X(x,t)\)^2+
\E \(X(y,s)\)^2-2\E\( X(x,t)X(y,s)\)\\
&=2C(0,0)-2C(x-y,t-s).\\
\end{split}
\end{equation*}
And
\begin{equation}\label{Eq:2C0}
\begin{split}
2C(0,0)-2C(x,t)&=2\si^2-\displaystyle\frac{2\si^2}{(1+a|t|^{2\a})^{\be
N/2}}\exp\Big(-\frac{c|x|^{2\ga}}{(1+a |t|^{2\a})^{\be\ga}}\Big)\\
&=2\si^2\displaystyle\frac{(1+a|t|^{2\a})^{\be N/2}-
\exp\Big(-\displaystyle\frac{c|x|^{2\ga}}{(1+a
|t|^{2\a})^{\be\ga}}\Big)}{(1+a|t|^{2\a})^{\be N/2}}.\\
\end{split}
\end{equation}
By using Taylor expansion, we can write (\ref{Eq:2C0}) as
\begin{equation*}
\begin{split}
&2\si^2\displaystyle\frac{1+\frac{\be
N}{2}a|t|^{2\a}+o(|t|^{2\a})-
1+\displaystyle\frac{c|x|^{2\ga}}{(1+a|t|^{2\a})^{\be\ga}}-
o\bigg(\displaystyle\frac{c|x|^{2\ga}}{(1+a|t|^{2\a})^{\be\ga}}\bigg)}
{(1+a|t|^{2\a})^{\be N/2}}\\
&=2\displaystyle\si^2\displaystyle\frac{\frac{\be
N}{2}a|t|^{2\a}+\displaystyle\frac{c|x|^{2\ga}}
{(1+a|t|^{2\a})^{\be\ga}}+o(|t|^{2\a})-
o\bigg(\displaystyle\frac{c|x|^{2\ga}}
{(1+a|t|^{2\a})^{\be\ga}}\bigg)}{(1+a|t|^{2\a})^{\be
N/2}}.
\end{split}
\end{equation*}
Hence we can find positive constants $c_{24} \le c_{25}$
such that
\begin{equation}\label{Eq:C5}
c_{24}\(|x|^{2\ga}+|t|^{2\a}\)\leq 2C(0,0)-2C(x,t)
\leq c_{25}\(|x|^{2\ga}+|t|^{2\a}\)
\end{equation}
for all $(x, t) \in \R^{d+1}$ with $|x|$ and $|t|$ small.
Replace $x$ and $t$ in (\ref{Eq:C5}) with $x-y$ and $t-s$
respectively, then (\ref{Eq:si2}) follows.

To prove (\ref{Eq:si3}), we make use of the fact that for any
Gaussian random vector $(U, V)$ with mean 0,
\[
{\rm Var}(U|V) = \frac{\big(\rho_{U, V}^2 - (\sigma_U^{\,} -
\sigma_V^{\,})^2\big) \big( (\sigma_U + \sigma_V)^2 -
\rho_{U, V}^2 \big)} {4 \sigma_V^2},
\]
where  $\rho_{U, V}^2= \E\big[(U-V)^2\big]$,
$\sigma_U^2 = \E(U^2)$ and $\sigma_V^2 = \E(V^2)$. Let $U= X(x,t)$
and $V = X(y,s)$, we derive
\[
\begin{split}
\Var\(X(x,t)|X(y,s)\)&=\frac{\big[C(0, 0) - C(x-y, t-s)\big]
\big[ C(0, 0) + C(x-y, t-s)\big]}{C(0, 0)}\\
&\ge c_{24}\(|x-y|^{2\ga}+|t-s|^{2\a}\).
\end{split}
\]
This proves (\ref{Eq:si3}). {\qed}

\vspace{.1in}
\noindent{\large \textbf{Proof of Proposition
\ref{Pr:Credim}}}

Eq. (\ref{Eq:Credimr}) follows from Proposition \ref{Le:Cre1} and
Theorem \ref{Th:dim}. Then let's prove  (\ref{Eq:CreGr1}),
where $0<\a\leq\ga \le 1$. By Proposition \ref{Le:Cre1} and
Theorem \ref{Th:dim}, we get
\begin{equation*}
 \dim{\rm Gr} \X([0,1]^{d+1})=\min_{1\leq k\leq
 d+1}\bigg\{\sum_{j=1}^k\frac{\overline{H}_k}{\overline{H}_j}
 +d+1-k+(1-\overline{H}_k)p;\,\,
 \sum_{j=1}^{d+1}\frac{1}{\overline{H}_j}\bigg\},
\end{equation*}
where $\overline{H}_1=\a,\,\overline{H}_2=\cdots
=\overline{H}_{d+1}=\ga$. Denote
\[
S(k)=\sum_{j=1}^k\frac{\overline{H}_k}{\overline{H}_j}
 +d+1-k+(1-\overline{H}_k)p.
\]
We have $S(1)=d+1+(1-\a)p$, $S(k)=d+\frac{\ga}{\a}+(1-\ga)p\triangleq
S$, for $2\leq k\leq d+1$. Also
$\sum_{j=1}^{d+1}\frac{1}{\overline{H}_j}=\frac{1}{\a}
+\frac{d}{\ga}$. Since
\begin{equation}\label{1}
\begin{split}
p<\frac{1}{\a}&\quad\Longleftrightarrow\quad
(\ga-\a)p<\frac{\ga-\a}{\a}\quad\Longleftrightarrow\quad S(1)<S\\
&\quad\Longleftrightarrow\quad
d+\frac{\ga}{\a}+(1-\ga)p<d+\frac{1}{\a}\quad\Longrightarrow\quad
S<\sum_{j=1}^{d+1}\frac{1}{\overline{H}_j},
\end{split}
\end{equation}
and
\begin{equation}\label{2}
p<\frac{1}{\a}+\frac{d}{\ga}\quad\Longleftrightarrow\quad
(1-\ga)p<\frac{1-\ga}{\a}+\frac{d(1-\ga)}{\ga}
\quad\Longleftrightarrow\quad
S<\sum_{j=1}^{d+1}\frac{1}{\overline{H}_j},
\end{equation}
one can see that (\ref{Eq:CreGr1}) follows from (\ref{1}) and (\ref{2}).

If $0<\ga\leq\a\le 1$, then $\overline{H}_1=\cdots
=\overline{H}_d=\ga$ and
$\overline{H}_{d+1}=\a$. So $S(d+1)=\frac{d\a}{\ga}+1+(1-\a)p$,
$S(k)=d+1+(1-\ga)p\triangleq \tilde{S}$, for $1\leq k\leq d$
and
$\sum_{j=1}^{d+1}\frac{1}{\overline{H}_j}
=\frac{1}{\a}+\frac{d}{\ga}$.
Similarly, by comparing these three terms, we will get
(\ref{Eq:CreGr2}). {\qed}

\vspace{.1in} \noindent{\large \textbf{Proof of Proposition
\ref{Pr:Credim2}}}

By Proposition \ref{Le:Cre1} and Theorem \ref{Th:dim2}
we get that when $\frac{1}{\a}+\frac{d}{\ga}<p$, for
every $x\in\p$, $\X^{-1}(x)=\emptyset$ a.s. And also, when
$\frac{1}{\a}+\frac{d}{\ga}>p$, then for any $x\in\p$,
with positive probability
\[
\dim\(\X^{-1}(x)\)=\min_{1\leq k\leq d+1}\bigg\{
\sum_{j=1}^k\frac{\overline{H}_k}{\overline{H}_j}+d+1-
k-\overline{H}_k\, p\bigg\}.
\]
If $0<\a\leq \ga <1$, we have
$\overline{H}_1=\a,\,\overline{H}_2= \cdots
=\overline{H}_{d+1}=\ga$. Denote
\[
T(k)=\sum_{j=1}^k\frac{\overline{H}_k}{\overline{H}_j}
+d+1-k-\overline{H}_k\, p,
\]
then $T(1)=d+1-\a p$, $T(k)=d+\frac{\ga}{\a}-\ga p\triangleq T$,
for $2\leq k\leq d+1$. Since $T(1)<T \, \Leftrightarrow\,
p<\frac{1}{\a},$ (\ref{Eq:Crel}) follows.

If $0<\ga\leq\a<1$, then $\overline{H}_1=\cdots=
\overline{H}_d=\ga$ and $\overline{H}_{d+1}=\a$.
It follows that $T(d+1)=\frac{d\a}{\ga}+1-\a p$
and $T(k)=d+1-\ga
p\triangleq \tilde{T}$, for $1\leq k\leq d$. Since
$$\tilde{T}<T(d+1) \ \Longleftrightarrow\
p<\frac{d}{\ga},$$
we derive (\ref{Eq:Crel2}). The proof is finished.
{\qed}

\vspace{.1in}
\bibliographystyle{plain}
\begin{small}

\end{small}

\end{document}